\documentclass{article}

\newif\ifarxiv
\arxivtrue

\usepackage{graphicx}
\usepackage{amsmath,amssymb,amsthm,bm}
\usepackage[a4paper]{geometry}
\usepackage{algorithm,algorithmic}
\usepackage[colorlinks=false,pdfborder={0 0 0}]{hyperref}
\usepackage{authblk}
\usepackage{color}
\usepackage{makecell}

\graphicspath{{fig/}}
\DeclareGraphicsExtensions{.mps,.pdf,.eps,.png}

\newenvironment{keywords}{\medskip\textbf{Keywords:}}{}
\newenvironment{AMS}{\medskip\textbf{AMS subject classifications (2020).}}{}

\newtheorem{theorem}{Theorem}

\newtheorem{remark}{Remark}
\newtheorem{lemma}{Lemma}

\theoremstyle{plain}

\DeclareMathOperator{\diag}{diag}

\newcommand{\fro}{\mathsf F}

\newcommand*{\trans}{^{\top}}

\newcommand*{\itrans}{^{-\top}}

\newcommand*{\pinv}{^{\dagger}}

\newcommand*{\abs}[1]{\bigl\lvert#1\bigr\rvert}
\newcommand*{\absbig}[1]{\biggl\lvert#1\biggr\rvert}
\newcommand*{\normbig}[1]{\left\Vert#1\right\rVert}
\newcommand*{\norm}[1]{\bigl\Vert#1\bigr\rVert}
\newcommand*{\normF}[1]{\bigl\Vert#1\bigr\rVert_{\fro}}

\def\adots{\mathinner{\mkern2mu\raise1pt\hbox{.}\mkern2mu
    \raise4pt\hbox{.}\mkern2mu\raise7pt\hbox{.}\mkern1mu}}

\newcommand*{\macheps}{\bm u}
\newcommand*{\machepsh}{\bm u_{h}}

\newcommand*{\epscholh}{\varepsilon^h_{\mathrm{chol}}}

\newcommand*{\epsmh}{\varepsilon^h_{\mathrm{M}}}

\newcommand*{\epseigh}{\varepsilon_{\mathrm{eig-tol}}}
\newcommand*{\epseigmh}{\varepsilon^h_{\mathrm{eig}}}

\newcommand*{\epsV}{\varepsilon_{\mathrm{V}}}

\newcommand*{\epsSVD}{\varepsilon_{\mathrm{SVD}}}
\newcommand*{\epstrsm}{\varepsilon_{\mathrm{trsm}}}
\newcommand*{\epsU}{\varepsilon_{\mathrm U}}

\newcommand*{\epssqrt}{\varepsilon_{\mathrm{sqrt}}}

\newcommand*{\working}{\mathtt{working}}
\newcommand*{\high}{\mathtt{higher}}

\newcommand*{\bigO}{O}

\title{Mixed precision thin SVD algorithms based on the Gram matrix}
\author[1]{Erin Carson}
\author[1]{Yuxin Ma}
\author[2]{Meiyue Shao}
\affil[1]{Department of Numerical Mathematics, Faculty of Mathematics and Physics, Charles University, Sokolovsk\'{a} 49/83, 186 75 Praha 8, Czechia

\texttt{Email: carson@karlin.mff.cuni.cz, yuxin.ma@matfyz.cuni.cz}}
\affil[2]{School of Data Science and MOE Key Laboratory for Computational Physical Sciences, Fudan University, Shanghai 200433, China

\texttt{Email: myshao@fudan.edu.cn}}

\begin{document}
\maketitle

\begin{abstract}
In this work, we present a mixed precision algorithm that leverages the Gram matrix and Jacobi methods to compute the singular value decomposition (SVD) of tall-and-skinny matrices.
By constructing the Gram matrix in higher precision and coupling it with a Jacobi algorithm, our theoretical analysis and numerical experiments both indicate that the singular values computed by this mixed precision thin SVD algorithm attain high relative accuracy.
In practice, our mixed precision thin SVD algorithm yields speedups of over \(10\times\) on a single CPU and about \(2\times\) on distributed memory systems when compared with traditional thin SVD methods.


\end{abstract}

\section{Introduction}
\label{sec:introduction}
Given a full-rank, tall-and-skinny matrix \(A \in \mathbb{R}^{m \times n}\) with \(m \gg n\), this work focuses on computing its singular value decomposition (SVD),
\begin{equation} \label{prob:svd}
    A = U \Sigma V\trans,
\end{equation}
where \(U \in \mathbb{R}^{m \times n}\) is an orthonormal matrix whose columns are the left singular vectors, \(V \in \mathbb{R}^{n \times n}\) is an orthogonal matrix whose columns are the right singular vectors, and \(\Sigma \in \mathbb{R}^{n \times n}\) is a diagonal matrix whose entries are the singular values.
This problem is encountered in many applications, such as principal component analysis and linear regression.

To compute the SVD of tall-and-skinny matrices, one typically first applies a thin QR factorization, \(A = QR\), which compresses the tall-and-skinny matrix \(A\) into an \(n \times n\) square matrix \(R\).
An SVD algorithm is then applied to \(R\), yielding \(R = U_R \Sigma V\trans\), from which the singular values and right singular vectors of \(A\) follow directly.
The left singular vectors of \(A\) can be obtained by \(U = Q U_R\).
For computing the SVD of the small matrix, the three most commonly used algorithms are the QR SVD, the divide-and-conquer (D\&C) SVD, and the one-sided Jacobi SVD.
As shown in~\cite{DV1992}, when computing the singular values of \(A = BD\) with the one-sided Jacobi SVD algorithm, one obtains high relative accuracy. Specifically, the accuracy of the computed singular values depends only on \(\kappa(B)\), rather than on \(\kappa(A)\), where the columns of \(B\) have unit norms and \(D\) is a diagonal matrix whose entries are the column \(2\)-norms of \(A\).
In contrast, for the QR SVD algorithm and the D\&C SVD algorithm, the accuracy of the results is governed by \(\kappa(A)\).

On modern computer architectures, communication is  increasingly costly relative to arithmetic operations.
Here, communication refers both to data transfer between different levels of the local memory hierarchy and to synchronization among parallel processing units.
In this QR-based thin SVD method, the QR factorization, which scales with \(m\), is the most expensive step in terms of both communication and computation, and for numerical stability, the initial reduction typically employs either Householder QR or tall-and-skinny QR.
In this paper, our objective is to develop an efficient thin SVD algorithm that guarantees high relative accuracy.
Therefore, we do not take into account randomized QR factorization algorithms, such as~\cite{GT2024}.
However, compared to the Cholesky QR algorithm, Householder QR and tall-and-skinny QR algorithms incur significantly higher communication costs when orthogonalizing a tall-and-skinny matrix.

The limitation of the Cholesky QR algorithm lies in its numerical instability~\cite{CLRT2022,YNYF2015,YTD2015}.
In~\cite{YTD2015}, the authors proposed a mixed precision Cholesky QR algorithm that improves numerical stability by computing both the Gram matrix and the Cholesky factorization in a precision higher than the working precision, referred to as higher precision.
Even though current hardware does not always natively support higher precision, the cost of computation is decreasing relative to communication on modern architectures.
Consequently, numerical experiments demonstrated that on a recent NVIDIA GPU, the mixed precision Cholesky QR algorithm was only \(1.4\times\) slower than the fixed precision one.
On the other hand,~\cite{M2025} investigated a truncated thin SVD algorithm based on the Gram matrix for the low-rank approximation problem.
This algorithm does not depend on higher precision but significantly improves the traditional error bounds involving \(\kappa^2(A)\), demonstrating that this Gram-based SVD algorithm is less unstable than it might initially appear.

Motivated by these ideas, we develop a mixed precision thin SVD algorithm that similarly depends on forming the Gram matrix in higher precision.
By appropriately combining it with either a two-sided Jacobi or a one-sided Jacobi SVD algorithm, our mixed precision thin SVD algorithm can deliver up to more than a \(10\times\) speedup on CPUs and roughly a \(2\times\) speedup on distributed memory systems compared to the QR-based thin SVD methods while still maintaining high relative accuracy.

In recent years, mixed precision algorithms have attracted much attention, especially in numerical linear algebra and high performance computing~\cite{Survey2021,HM2022}.
In addressing the SVD problem, several other studies have employed mixed-precision techniques. The works 
\cite{GMS2025,SXZ2026} mainly utilized lower precision (i.e., a precision lower than the working precision) to accelerate computations while maintaining accuracy.
In~\cite{ZTW2026}, the authors developed a mixed-precision one-sided Jacobi SVD algorithm that incorporates three precision levels (lower, working, and higher) in order to enhance the accuracy of the computed singular values.
In contrast to these approaches, our goal is to exploit higher precision to guarantee high relative accuracy while still achieving significant performance improvements.

The remainder of this paper is organized as follows.
Section~\ref{sec:algorithm} introduces our Gram-based mixed precision thin SVD algorithm.
The high relative accuracy of this mixed precision algorithm is demonstrated in Section~\ref{sec:accuracy}.
In Section~\ref{sec:experiments}, we report numerical experiments that illustrate both the accuracy and the efficiency of our mixed precision thin SVD algorithms.

A few remarks regarding notation are necessary before proceeding.
In general, we use uppercase Roman letters such as \(A\) and \(Q\) to represent matrices, and \(A(i, :~)\) to refer to the \(i\)-th row of a matrix.
We write $\norm{\cdot}$ for the spectral norm and $\normF{\cdot}$ for the Frobenius norm.
The symbols \(\lambda_{\max}\) and \(\lambda_{\min}\) denote the largest and smallest eigenvalues, respectively.
We use \(\sigma_i(\cdot)\) to denote the singular values, ordered as \(\sigma_1(\cdot)\geq\sigma_2(\cdot)\geq\dotsi\geq\sigma_n(\cdot)\).
The condition number of a matrix is written as \(\kappa(\cdot) = \sigma_{\max}(\cdot)/\sigma_{\min}(\cdot)\).
We denote the unit roundoff in the working precision by \(\macheps\), and in a higher precision by \(\machepsh\).
Computed quantities are indicated with a hat.
For instance, $\hat{Q}$ denotes the finite-precision approximation of the orthonormal matrix $Q$.

\section{Mixed precision thin SVD algorithm}
\label{sec:algorithm}

In this section, we introduce a mixed precision algorithm for solving Problem~\eqref{prob:svd} based on the following steps:
\begin{enumerate}
    \item Compute the Gram matrix \(M = A\trans A\). \label{step:Gram}
    \item Compute the spectral decomposition of \(M\), that is, \(M = V \Sigma^2 V\trans\), where \(V\) is an orthogonal matrix whose columns are the eigenvectors of \(M\) and are also the right singular vectors of \(A\), and \(\Sigma\) is a diagonal matrix whose entries are the singular values of \(A\). \label{step:eigen}
    \item Compute the left singular vectors \(U =  AV\Sigma^{-1}\).%
    \footnote{Note that we only consider the situation where \(A\) has full column rank. In this case, both \(A\pinv\) and \(\Sigma^{-1}\) are well-defined.}
\end{enumerate}
Unlike the thin SVD algorithm that relies on QR factorization, the main computational cost of the above procedure lies in forming the Gram matrix.
As a result, it can benefit from the higher efficiency of matrix multiplication compared to QR factorization.
However, a drawback of this approach is the instability introduced by the Gram matrix—it squares the condition number of the problem.
A natural remedy is to carry out Steps~\ref{step:Gram} and~\ref{step:eigen}, which involve the Gram matrix, in higher precision.
We outline this mixed precision algorithm, which relies on a higher-precision Gram matrix, in Algorithm~\ref{alg:mpthinsvd}.
Since matrix multiplication is significantly faster than QR factorization, we can still expect an overall speedup of Algorithm~\ref{alg:mpthinsvd} over the QR-based thin SVD method for \(m \gg n\), even when the matrix multiplication is performed in higher precision, as will be demonstrated in Section~\ref{sec:experiments}.

\begin{algorithm}[!tb]
\begin{algorithmic}[1]
    \caption{Mixed precision thin SVD algorithm based on the Gram matrix  \label{alg:mpthinsvd}}
    \REQUIRE
     A matrix \(A \in \mathbb R^{m\times n}\), the working precision \(\macheps\) and the higher precision \(\machepsh\leq \macheps\).
    \ENSURE
    Computed orthonormal matrix \(U\in \mathbb R^{m\times n}\), orthogonal matrix \(V\in \mathbb R^{n\times n}\), and diagonal matrix \(\Sigma\in \mathbb R^{n\times n}\) approximating \(A = U\Sigma V\trans\).
    
    \STATE Cast \(A\) to the higher precision, i.e., \(A_h\gets \high(A)\).
    \STATE Compute the Gram matrix \(M_h = A_h\trans A_h\) at precision \(\machepsh\).
    \STATE Compute the spectral decomposition of the symmetric positive definite matrix \(M_h\) at precision \(\machepsh\), i.e., \(M_h = V_h\Sigma_h^2V_h\trans\), where the orthogonal matrix \(V_h\) and the diagonal matrix \(\Sigma_h\) contain the eigenvectors and the corresponding eigenvalues, respectively. \label{line:eigen}
    \STATE Cast \(\Sigma_h\) back to the working precision to get \(\Sigma\), i.e. \(\Sigma\gets\working(\Sigma_h)\).
    \STATE Cast \(V_h\) back to the working precision, i.e., \(V\gets \working(V_h)\).
    \STATE Obtain \(U\) from \(U\gets A(V\Sigma^{-1})\) at working precision \(\macheps\). \label{line:computeU}
\end{algorithmic}
\end{algorithm}

\subsection{Eigensolver used in Line~\ref{line:eigen} of Algorithm~\ref{alg:mpthinsvd}}
In Line~\ref{line:eigen} of Algorithm~\ref{alg:mpthinsvd}, the spectral decomposition can be computed using various methods, such as the QR algorithm, the D\&C algorithm, or the two-sided Jacobi algorithm.
It should be noted that the computational cost of this spectral decomposition is relatively low, as it does not depend on \(m\).
Therefore, we can employ a more stable and accurate method in Line~\ref{line:eigen}, even if it is somewhat more computationally expensive.
Compared with the QR algorithm and the D\&C algorithm, the eigenvalues computed by the two-sided Jacobi algorithm exhibit high relative accuracy.
As will be discussed in Section~\ref{sec:accuracy}, the singular values computed by Algorithm~\ref{alg:mpthinsvd} when employing the two-sided Jacobi algorithm as the eigensolver can achieve high relative accuracy comparable to that of the one-sided Jacobi SVD algorithm~\cite{DV2008a}.
Therefore, the two-sided Jacobi algorithm is a suitable choice for Line~\ref{line:eigen} in Algorithm~\ref{alg:mpthinsvd}.

Since the two-sided Jacobi algorithm is not available in LAPACK, one option is to employ the one-sided Jacobi SVD algorithm proposed in~\cite{DV2008a,DV2008b} instead, as the eigensolver takes a symmetric positive definite matrix as input.
We additionally present an alternative approach to compute the spectral decomposition of the Gram matrix, summarized in Algorithm~\ref{alg:eigen}, which also preserves the high relative accuracy of the results produced by Algorithm~\ref{alg:mpthinsvd}.

\begin{algorithm}[!tb]
\begin{algorithmic}[1]
    \caption{Mixed precision symmetric eigensolver for the Gram matrix  \label{alg:eigen}}
    \REQUIRE
     The working precision \(\macheps\), the higher precision \(\machepsh\leq \macheps\), and a Gram matrix \(M_h \in \mathbb R^{n\times n}\) stored at higher precision \(\machepsh\).
    \ENSURE
    Computed orthogonal matrix \(V\in \mathbb R^{n\times n}\) and diagonal matrix \(\Sigma_h\in \mathbb R^{n\times n}\) approximating \(M_h = V\Sigma^2V\trans\), where \(V\) and \(\Sigma\) are at working precision \(\macheps\).

    \STATE Compute the Cholesky factorization of the Gram matrix \(M_h\), i.e., \(M_h = R_h\trans R_h\) at precision \(\machepsh\), where \(R_h\) is upper triangular.
    \STATE Cast \(R_h\) back to the working precision, i.e., \(R\gets \working(R_h)\).
    \STATE Compute the SVD of \(R\) at precision \(\macheps\), i.e., \(R = U_R\Sigma V\trans\), where \(U_R\) and \(V\) are orthogonal matrices, and \(\Sigma\) is a diagonal matrix.
\end{algorithmic}
\end{algorithm}

\section{Accuracy}
\label{sec:accuracy}
In this section, we will prove the backward stability, the high relative accuracy of the computed singular values, and the orthogonality of the computed singular vectors for Algorithm~\ref{alg:mpthinsvd}.
According to the standard rounding error analysis, each step of Algorithm~\ref{alg:mpthinsvd} satisfies that, respectively,
\begin{alignat}{2}
    \hat{M}_h &= A\trans A + \Delta M, \quad &&\abs{\Delta M} \leq \epsmh\abs{A}\trans\abs{A}, \label{eq:ATA}\\
    \tilde{V}_h\trans(\hat{M}_h+\Delta\hat{M}_h)\tilde{V}_h &= \hat{S}_h + \Delta S,\quad 
    && \norm{D^{-1}\Delta \hat{M}_hD^{-1}} \leq \epseigmh \norm{B}^2 \label{eq:eig} \\
    &  &&\text{and}\quad\abs{\Delta S}_{ij} \leq \epseigh(\tilde{\Sigma}_h)_{ii} (\tilde{\Sigma}_h)_{jj}, \label{eq:eig-2} \\
    & && \text{with}\quad (\tilde{\Sigma}_h)^2 =\hat{S}_h,\quad \forall i, j\leq n, \label{eq:Sh=Sigma2} \\
    \hat{\sigma}_{i} &= (\tilde{\Sigma}_h)_{ii}(1+\delta \sigma_i), \quad
    &&\abs{\delta \sigma_i} \leq \epssqrt, \label{eq:sqrt}\\
    \hat{V} &= \tilde{V}_h + \Delta V, \quad &&\abs{\Delta V}\leq \epsV \abs{\tilde{V}_h}, \label{eq:castV} \\
    \hat{U} &= A\hat{V}\hat{\Sigma}^{-1}+\Delta U, \quad &&\abs{\Delta 
    U}\leq \epsU\abs{A}\abs{\hat{V}}\hat{\Sigma}^{-1}, \label{eq:AV-0} 
\end{alignat}
where \(A = BD\), \(\hat{S}_h\) is a diagonal matrix, and \(\tilde{V}_h\) is an orthogonal matrix.
Here, \(\epsmh\) and \(\epseigmh\) are parameters that depend on \(\machepsh\), while \(\epseigh\), \(\epssqrt\), \(\epsV\), and \(\epsU\) are functions of \(\macheps\).
Moreover, \(\epseigmh\) and \(\epseigh\) are determined by the eigensolver employed in Line~\ref{line:eigen} of Algorithm~\ref{alg:mpthinsvd}.
For example, according to~\cite[Theorem~3.1]{DV1992}, when the two-sided Jacobi algorithm is used as the eigensolver, we have \(\epseigmh=\bigO(\machepsh)\) and \(\epseigh=\bigO(\macheps)\), where this latter quantity represents the tolerance selected for the stopping criterion.
For readability, we summarize these parameters in Table~\ref{tab:notation}.

\begin{table}[!tb]
\centering
\caption{Notations of rounding error analysis.
Here \(A\in\mathbb{R}^{m\times n}\); \(S\), \(\Sigma\in\mathbb{R}^{n\times n}\)
are diagonal matrices; \(\tilde{V}\in\mathbb{R}^{n\times n}\) is an orthogonal matrix satisfying~\eqref{eq:eig}; \(\hat{V}\in\mathbb{R}^{n\times n}\) denotes the quantity in working precision that is obtained by converting the higher precision result \(\hat{V_h}\) back to the working precision format.}
\label{tab:notation}
\footnotesize
\begin{tabular}{ccccc}
\hline Notation & Sources & Definitions
& \makecell[c]{Upper bound\\ on \(\varepsilon\)} & References \\
\hline
\(\epsmh\)  & \(M_h = A\trans A\) & \eqref{eq:ATA}
& \(\bigO(\machepsh)\) & \cite[Section 3.5]{H2002}\\
\(\epseigmh\)  & \(M_h = V_hSV_h\trans\) & \eqref{eq:eig}
& \(\bigO(\machepsh)\) & \cite[Theorem~3.1]{DV1992}\\
\(\epseigh\)  & \(M_h = V_hSV_h\trans\) & \eqref{eq:eig-2}
& \(\bigO(\macheps)\) & \cite[Theorem~3.1]{DV1992}\\
\(\epssqrt\)  & \(\Sigma^2=S\) & \eqref{eq:sqrt}
& \(\bigO(\macheps)\)    & \cite[Section 1.1]{H2002}\\
\(\epsV\)  & the distance between \(\hat{V}\) and \(\tilde{V}_h\) & \eqref{eq:castV}
& \(\bigO(\macheps)\)    & \cite[Theorem~3.1]{DV1992}\\
\(\epsU\)  & \(U = AV\Sigma^{-1}\) & \eqref{eq:AV-0}
& \(\bigO(\macheps)\) & \cite[Section 3.5]{H2002} \\
\hline
\end{tabular}
\end{table}

Combining~\eqref{eq:castV} and~\eqref{eq:AV-0}, we have
\begin{equation} \label{eq:AV}
    \hat{U} = A\tilde{V}_h\hat{\Sigma}^{-1}+\Delta \tilde{U}, \quad \abs{\Delta \tilde{U}}\leq \bigl(\epsV + \epsU(1+\epsV)\bigr)\abs{A}\abs{\tilde{V}_h}\hat{\Sigma}^{-1}.
\end{equation}

In Section~\ref{subsec:accuracy-algo1}, we first demonstrate that Algorithm~\ref{alg:mpthinsvd}, when equipped with an eigensolver (in Line~\ref{line:eigen}) satisfying~\eqref{eq:eig} and~\eqref{eq:eig-2}, is numerically stable and produces singular values with high relative accuracy, comparable to those obtained by the one-sided Jacobi SVD algorithm in~\cite{DV1992}.
Subsequently, we show that the eigensolver described in Algorithm~\ref{alg:eigen} indeed satisfies the requirements~\eqref{eq:eig} and~\eqref{eq:eig-2} in Section~\ref{subsec:accuracy-eigen}.

\subsection{Accuracy of Algorithm~\ref{alg:mpthinsvd}}
\label{subsec:accuracy-algo1}
In Theorem~\ref{thm:algo1}, we give the backward stability, high relative accuracy of the computed singular values, and the orthogonality of the computed singular vectors of Algorithm~\ref{alg:mpthinsvd} with a specific eigensolver (in Line~\ref{line:eigen}) satisfying~\eqref{eq:eig} and~\eqref{eq:eig-2}.

\begin{theorem} \label{thm:algo1}
    Assume that \(\hat{\sigma}_i\) with \(i\leq n\) are the computed singular values of \(A = BD\) computed by Algorithm~\ref{alg:mpthinsvd}, where the columns of \(B\) have unit norms and \(D\) is a diagonal matrix containing the column norms of \(A\).
    Also, assume that~\eqref{eq:ATA}--\eqref{eq:AV-0} are satisfied with given \(\epsmh\), \(\epseigmh\), \(\epseigh\), \(\epssqrt\), \(\epsV\), and \(\epsU\).
    Then
    \begin{equation} \label{eq:backward}
        \hat{U}\hat{\Sigma}\tilde{V}_h\trans = A+\Delta A, \quad \norm{\Delta A(i, :)}\leq \sqrt{n}\bigl(\epsV + \epsU(1+\epsV)\bigr)\norm{A(i, :)}.
    \end{equation}
    If it further holds that
    \begin{equation} \label{eq:thm:assump}
        \varepsilon_1 := 2\epseigh + 2\bigl(n^2\epsmh + \epseigmh\bigr)\kappa^2(B) 
        + 4n\sqrt{n}\bigl(\epsV + \epsU(1+\epsV)\bigr)\kappa(B) \leq \frac{1}{2},
    \end{equation}
    then
    \begin{equation} \label{eq:thm:highaccuracy_sv}
        \frac{\abs{\hat{\sigma}_i-\sigma_i(A)}}{\sigma_i(A)}
        \leq 2\epssqrt + 2\epseigh + 4\bigl(n^2\epsmh+\epseigmh\bigr)\kappa^2(B)
    \end{equation}
    and
    \begin{equation} \label{eq:thm:orthU1}
    \normF{\hat{U}\trans \hat{U}-I}
    \leq \frac{2\,\sqrt{n}\epssqrt + n\varepsilon_1}{1-2\epssqrt}.
\end{equation}
\end{theorem}

\begin{proof}
In the following, we will prove~\eqref{eq:backward}, \eqref{eq:thm:highaccuracy_sv}, and~\eqref{eq:thm:orthU1}, respectively.

\paragraph{Proof of rowwise backward stability}
First, we aim to prove the rowwise backward stability, i.e., \eqref{eq:backward}, which is directly derived from~\eqref{eq:AV} since \(\Delta A = \Delta \tilde{U}\hat{\Sigma}\tilde{V}_h\trans\) satisfies
\begin{equation*}
\begin{split}
\norm{\Delta A(i, :)} &\leq \bigl(\epsV + \epsU(1+\epsV)\bigr)\norm{A(i, :)} \cdot\normbig{\abs{\tilde{V}_h}\abs{\tilde{V}_h}\trans}\\
&\leq \sqrt{n}\bigl(\epsV + \epsU(1+\epsV)\bigr)\norm{A(i, :)}
\end{split}
\end{equation*}
by~\cite[Lemma~6.6]{H2002} and the fact that \(\tilde{V}_h\) is orthogonal.

\paragraph{Proof of high relative accuracy of computed singular values}
Then our aim is to prove the high relative accuracy of computed singular values, i.e., \eqref{eq:thm:highaccuracy_sv}.

Substituting \(\hat{M}_h\) involved in~\eqref{eq:eig} by~\eqref{eq:ATA}, \eqref{eq:eig} can be rewritten as
\begin{equation} \label{eq:VhAAVh}
    \tilde{V}_h\trans A\trans A\tilde{V}_h = \hat{S}_h + \Delta S + \bigl(- \tilde{V}_h\trans \Delta M\tilde{V}_h\bigr) + \bigl(-\tilde{V}_h\trans \Delta\hat{M}_h\tilde{V}_h\bigr).
\end{equation}
We then bound \(-\tilde{V}_h\trans \Delta M\tilde{V}_h\) and \(-\tilde{V}_h\trans \Delta\hat{M}_h\tilde{V}_h\) separately.

From~\eqref{eq:ATA} and \(A=BD\), we have
\begin{equation} \label{eq:VhTDMVh}
\begin{split}
    \abs{\tilde{V}_h\trans \Delta M\tilde{V}_h}_{ij}
    &= \abs{\tilde{V}_h\trans (A\pinv A)\trans \Delta M (A\pinv A)\tilde{V}_h}_{ij} \\
    &\leq \norm{A\tilde{V}_h(:, i)}\cdot\normbig{\abs{A\pinv}\trans\abs{\Delta M}\abs{A\pinv}}\cdot\norm{A\tilde{V}_h(:, j)} \\
    &\leq \norm{A\tilde{V}_h(:, i)}\cdot\normbig{\abs{B\pinv}\trans D^{-1}\abs{\Delta M}D^{-1}\abs{B\pinv}}\cdot\norm{A\tilde{V}_h(:, j)} \\
    &\leq \epsmh\norm{A\tilde{V}_h(:, i)}\cdot\normbig{\abs{B\pinv}\trans \abs{B}\trans \abs{B}\abs{B\pinv}}\cdot\norm{A\tilde{V}_h(:, j)} \\
    &\leq n^2\epsmh\kappa^2(B) \norm{A\tilde{V}_h(:, i)}\cdot \norm{A\tilde{V}_h(:, j)}.
\end{split}
\end{equation}
Similarly, by~\eqref{eq:eig}, \(\abs{\tilde{V}_h\trans \Delta \hat{M}_h\tilde{V}_h}_{ij}\) is bounded by
\begin{equation} \label{eq:VhTDMhVh-0}
\begin{split}
    \abs{\tilde{V}_h\trans \Delta \hat{M}_h\tilde{V}_h}_{ij}
    &\leq \norm{A\tilde{V}_h(:, i)}\cdot\norm{(B\pinv)\trans D^{-1} \Delta \hat{M}_h D^{-1}B\pinv}\cdot\norm{A\tilde{V}_h(:, j)} \\
    &\leq \epseigmh \norm{A\tilde{V}_h(:, i)}\cdot\norm{B}^2\cdot\norm{B\pinv}^2\cdot\norm{A\tilde{V}_h(:, j)} \\
    &= \epseigmh \kappa^2(B) \norm{A\tilde{V}_h(:, i)}\cdot\norm{A\tilde{V}_h(:, j)}.
\end{split}
\end{equation}
Combining~\eqref{eq:VhAAVh} with~\eqref{eq:VhTDMVh} and~\eqref{eq:VhTDMhVh-0}, \(\norm{A\tilde{V}_h(:, i)}\), involved in the right-hand sides of~\eqref{eq:VhTDMVh} and~\eqref{eq:VhTDMhVh-0}, can be bounded by
\begin{equation*}
\begin{split}
    \norm{A\tilde{V}_h(:, i)}^2 &= \bigl(\tilde{V}_h\trans A\trans A\tilde{V}_h\bigr)_{ii} \\
    &\leq (\hat{S}_h)_{ii}
    + \abs{\tilde{V}_h\trans \Delta M\tilde{V}_h}_{ii}
    + \abs{\tilde{V}_h\trans \Delta \hat{M}_h\tilde{V}_h}_{ii} \\
    &\leq (\hat{S}_h)_{ii} + \bigl(n^2\epsmh + \epseigmh\bigr)\kappa(B)^2 \norm{A\tilde{V}_h(:, i)}^2,
\end{split}
\end{equation*}
which implies that
\begin{equation} \label{eq:norm-AVi}
    \norm{A\tilde{V}_h(:, i)} \leq \sqrt{\frac{(\hat{S}_h)_{ii}}{1-\bigl(n^2\epsmh + \epseigmh\bigr)\kappa(B)^2}}
    \leq \sqrt{2}\,(\tilde{\Sigma}_h)_{ii}.
\end{equation}

Let \(\Delta E = \tilde{V}_h\trans A\trans A\tilde{V}_h - \hat{S}_h\).
We further obtain
\begin{equation*}
    \norm{\tilde{\Sigma}_h^{-1}\Delta E\tilde{\Sigma}_h^{-1}}\leq \epseigh + 2\bigl(n^2\epsmh+\epseigmh\bigr)\kappa^2(B).
\end{equation*}
Applying~\cite[Theorem 2.3]{DV1992} to \(\hat{S}_h + \Delta E\) with \(\hat{S}_h = \tilde{\Sigma}_h^2\) and
\[
\tilde{V}_h\trans A\trans A\tilde{V}_h = \hat{S}_h + \Delta E = \tilde{\Sigma}_h\bigl(I+ \tilde{\Sigma}_h^{-1}\Delta E\tilde{\Sigma}_h^{-1}\bigr)\tilde{\Sigma}_h,
\]
we have
\begin{equation} \label{eq:bound-lambdaXTX-Sh}
    \absbig{\frac{\lambda_i(A\trans A) - (\hat{S}_h)_{ii}}{(\hat{S}_h)_{ii}}} 
    = \absbig{\frac{\lambda_i(\tilde{V}_h\trans A\trans A\tilde{V}_h) - (\hat{S}_h)_{ii}}{(\hat{S}_h)_{ii}}}
    \leq \epseigh + 2\bigl(n^2\epsmh+\epseigmh\bigr)\kappa^2(B).
\end{equation}
Combining~\eqref{eq:bound-lambdaXTX-Sh} with
\begin{equation*}
\begin{split}
    \absbig{\frac{\lambda_i(A\trans A) - (\hat{S}_h)_{ii}}{(\hat{S}_h)_{ii}}}
    &= \absbig{\frac{\sigma^2_i(A) - (\tilde{\Sigma}_h)^2_{ii}}{(\tilde{\Sigma}_h)^2_{ii}}} \\
    &= \absbig{\frac{\sigma_i(A) - (\tilde{\Sigma}_h)_{ii}}{(\tilde{\Sigma}_h)_{ii}}}
    \cdot\frac{\sigma_i(A) + (\tilde{\Sigma}_h)_{ii}}{(\tilde{\Sigma}_h)_{ii}} \\
    &\geq \absbig{\frac{\sigma_i(A) - (\tilde{\Sigma}_h)_{ii}}{(\tilde{\Sigma}_h)_{ii}}},
\end{split}
\end{equation*}
we further have
\begin{equation} \label{eq:proof-highaccuracy:gap-sigmaA-Sigmah}
    \absbig{\frac{\sigma_i(A) - (\tilde{\Sigma}_h)_{ii}}{(\tilde{\Sigma}_h)_{ii}}}
    \leq \epseigh + 2\bigl(n^2\epsmh+\epseigmh\bigr)\kappa^2(B).
\end{equation}
Furthermore, from the assumption~\eqref{eq:thm:assump}, \eqref{eq:proof-highaccuracy:gap-sigmaA-Sigmah} implies
\begin{equation*}
    \frac{\sigma_i(A)}{(\tilde{\Sigma}_h)_{ii}} \geq 1 - \absbig{\frac{\sigma_i(A) - (\tilde{\Sigma}_h)_{ii}}{(\tilde{\Sigma}_h)_{ii}}}
    \geq 1- \biggl(\epseigh + 2\bigl(n^2\epsmh+\epseigmh\bigr)\kappa^2(B)\biggr)
    \geq \frac{1}{2}
\end{equation*}
and then
\begin{equation} \label{eq:proof-highaccuracy:bound-Sigmah/sigmaA}
    \frac{(\tilde{\Sigma}_h)_{ii}}{\sigma_i(A)}\leq 2.
\end{equation}
Combining~\eqref{eq:proof-highaccuracy:gap-sigmaA-Sigmah} with~\eqref{eq:proof-highaccuracy:bound-Sigmah/sigmaA} and~\eqref{eq:sqrt}, we can conclude the proof of~\eqref{eq:thm:highaccuracy_sv} since
\begin{equation*}
\begin{split}
    \absbig{\frac{\sigma_i(A) - \hat{\sigma}_i}{\sigma_i(A)}}
    &= \absbig{\frac{\sigma_i(A) - (\tilde{\Sigma}_h)_{ii}(1+\delta\sigma_i)}{\sigma_i(A)}} \\
    &\leq \absbig{\frac{\sigma_i(A) - (\tilde{\Sigma}_h)_{ii}}{\sigma_i(A)}}
    + \abs{\delta\sigma_i}\frac{(\tilde{\Sigma}_h)_{ii}}{\sigma_i(A)} \\
    &\leq \absbig{\frac{\sigma_i(A) - (\tilde{\Sigma}_h)_{ii}}{(\tilde{\Sigma}_h)_{ii}}}
    \frac{(\tilde{\Sigma}_h)_{ii}}{\sigma_i(A)} + \abs{\delta\sigma_i}\frac{(\tilde{\Sigma}_h)_{ii}}{\sigma_i(A)} \\
    &\leq 2\epssqrt + 2\epseigh + 4\bigl(n^2\epsmh+\epseigmh\bigr)\kappa^2(B).
\end{split}
\end{equation*}

\paragraph{Proof of orthonormality of \(\hat{U}\)}
It remains to show that \(\hat{U}\) is close to orthonormal.
From~\eqref{eq:AV}, we directly have
\begin{equation} \label{eq:AV-1}
    X := \hat{U}\hat{\Sigma} = A\tilde{V}_h+\Delta X, \quad \abs{\Delta X}\leq \bigl(\epsV + \epsU(1+\epsV)\bigr)\abs{A}\abs{\tilde{V}_h}.
\end{equation}
Furthermore, \(\abs{X\trans X}\) can be bounded as follows:
\begin{equation}  \label{eq:XTX-0}
\begin{split}
    \abs{X\trans X}_{ij}
    &= \abs{(A\tilde{V}_h + \Delta X)\trans (A\tilde{V}_h + \Delta X)}_{ij} \\
    &\leq \abs{\tilde{V}_h\trans A\trans A\tilde{V}_h}_{ij} + \abs{\tilde{V}_h\trans A\trans\Delta X}_{ij} + \abs{\tilde{V}_h\trans A\trans\Delta X}_{ji} + \abs{\Delta X\trans \Delta X}_{ij} \\
    &\leq \abs{\tilde{V}_h\trans (\hat{M}_h + \Delta \hat{M}_h)\tilde{V}_h}_{ij}
    + \abs{\tilde{V}_h\trans \Delta M\tilde{V}_h}_{ij} 
    + \abs{\tilde{V}_h\trans \Delta \hat{M}_h\tilde{V}_h}_{ij} \\
    &\quad+ \abs{\tilde{V}_h\trans A\trans\Delta X}_{ij}
    + \abs{\tilde{V}_h\trans A\trans\Delta X}_{ji}
    + \abs{\Delta X\trans \Delta X}_{ij} \\
    &\leq \abs{\hat{S}_h}_{ij} + \abs{\Delta S}_{ij}
    + \abs{\tilde{V}_h\trans \Delta M\tilde{V}_h}_{ij}
    + \abs{\tilde{V}_h\trans \Delta \hat{M}_h\tilde{V}_h}_{ij} \\
    &\quad+ \abs{\tilde{V}_h\trans A\trans\Delta X}_{ij}
    + \abs{\tilde{V}_h\trans A\trans\Delta X}_{ji}
    + \abs{\Delta X\trans \Delta X}_{ij}.
\end{split}
\end{equation}
Note that \(\abs{\Delta S}\), \(\tilde{V}_h\trans \Delta M\tilde{V}_h\), and \(\tilde{V}_h\trans \Delta \hat{M}_h\tilde{V}_h\) have been bounded in~\eqref{eq:eig-2}, \eqref{eq:VhTDMVh}, and~\eqref{eq:VhTDMhVh-0}, respectively.
Then we will treat other terms involved in the bound of~\eqref{eq:XTX-0} separately.

From~\eqref{eq:AV-1} and noticing the fact \(\tilde{V}_h\trans A\pinv A\tilde{V}_h = I\), we have
\begin{equation} \label{eq:VhTATDX-0}
\begin{split}
    \abs{\tilde{V}_h\trans A\trans \Delta X}_{ij}
    &\leq \bigl(\epsV + \epsU(1+\epsV)\bigr) \bigl(\abs{A\tilde{V}_h}\trans \abs{A}\abs{\tilde{V}_h}\bigr)_{ij} \\
    &= \bigl(\epsV + \epsU(1+\epsV)\bigr) \bigl(\abs{A\tilde{V}_h}\trans \abs{A}\abs{\tilde{V}_h}\cdot\tilde{V}_h\trans A\pinv A\tilde{V}_h\bigr)_{ij} \\
    &\leq \bigl(\epsV + \epsU(1+\epsV)\bigr) \norm{A\tilde{V}_h(:, i)}\cdot \norm{\abs{A}\abs{\tilde{V}_h}\tilde{V}_h\trans A\pinv}\cdot \norm{A\tilde{V}_h(:, j)}.
\end{split}
\end{equation}
Notice that
\[
A\pinv = (A\trans A)^{-1}A\trans
= (DB\trans BD)^{-1}DB\trans
= D^{-1}(B\trans B)^{-1}B\trans
= D^{-1} B\pinv,
\]
which implies that \(\abs{A}\abs{\tilde{V}_h}\tilde{V}_h\trans A\pinv = \abs{B}D\abs{\tilde{V}_h}\tilde{V}_h\trans D^{-1}B\pinv\).
By noticing that \(\tilde{V}_h\) is orthogonal, it follows that
\begin{equation} \label{eq:AVVA-0}
\begin{split}
    \normbig{\abs{A}\abs{\tilde{V}_h}\tilde{V}_h\trans A\pinv}
    &\leq \sqrt{n}\norm{B}\cdot \normbig{D\abs{\tilde{V}_h}\tilde{V}_h\trans D^{-1}}\cdot \norm{B\pinv} \\
    &\leq \sqrt{n}\norm{B}\cdot \normbig{D\abs{\tilde{V}_h}\abs{\tilde{V}_h}\trans D^{-1}}\cdot \norm{B\pinv}. 
\end{split}
\end{equation}
In the right-hand side of~\eqref{eq:AVVA-0}, note that \(\abs{\tilde{V}_h}\abs{\tilde{V}_h}\trans\) is a symmetric matrix, then we have
\begin{equation*}
    \normbig{D\abs{\tilde{V}_h}\abs{\tilde{V}_h}\trans D^{-1}}
    = \lambda_1(D\abs{\tilde{V}_h}\abs{\tilde{V}_h}\trans D^{-1})
    = \lambda_1(\abs{\tilde{V}_h}\abs{\tilde{V}_h}\trans)
    = \norm{\abs{\tilde{V}_h}\abs{\tilde{V}_h}\trans}.
\end{equation*}
Together with the fact that \(\tilde{V}_h\) is orthogonal, using~\cite[Lemma~6.6]{H2002} we have
\begin{equation} \label{eq:norm-DVhVhTDinv}
    \normbig{D\abs{\tilde{V}_h}\abs{\tilde{V}_h}\trans D^{-1}}
    = \normbig{\abs{\tilde{V}_h}\abs{\tilde{V}_h}\trans}
    \leq n\norm{\tilde{V}_h}^2 = n.
\end{equation}
Combining~\eqref{eq:AVVA-0} with~\eqref{eq:norm-DVhVhTDinv} and \(\kappa(B) = \norm{B}\norm{B\pinv}\), we have
\begin{equation} \label{eq:AVVA}
    \normbig{\abs{A}\abs{\tilde{V}_h}\tilde{V}_h\trans A\pinv}
    \leq n\sqrt{n}\,\kappa(B).
\end{equation}
Using~\eqref{eq:AVVA} to substitute \(n\sqrt{n}\kappa(B)\) for \(\normbig{\abs{A}\abs{\tilde{V}_h}\tilde{V}_h\trans A\pinv}\) in the bound of~\eqref{eq:VhTATDX-0}, it follows that
\begin{equation} \label{eq:VhTATDX}
\begin{split}
    \abs{\tilde{V}_h\trans A\trans \Delta X}_{ij}
    &\leq n\sqrt{n}\bigl(\epsV + \epsU(1+\epsV)\bigr)\kappa(B) \norm{A\tilde{V}_h(:, i)} \cdot \norm{A\tilde{V}_h(:, j)}.
\end{split}
\end{equation}
Similarly, \(\abs{\Delta X\trans\Delta X}\) can be bounded as follows:
\begin{equation} \label{eq:DXTDX}
    \abs{\Delta X\trans\Delta X}_{ij} \leq n^3\bigl(\epsV + \epsU(1+\epsV)\bigr)^2\kappa^2(B) \norm{A\tilde{V}_h(:, i)} \cdot \norm{A\tilde{V}_h(:, j)}.
\end{equation}

Substituting the terms \(\abs{\Delta S}_{ij}\), \(\abs{\tilde{V}_h\trans \Delta M\tilde{V}_h}_{ij}\), \(\abs{\tilde{V}_h\trans \Delta \hat{M}_h\tilde{V}_h}_{ij}\), \(\abs{\tilde{V}_h\trans A\trans\Delta X}_{ij}\), and \(\abs{\Delta X\trans \Delta X}_{ij}\) involved in~\eqref{eq:XTX-0} by~\eqref{eq:eig}, \eqref{eq:VhTDMVh}, \eqref{eq:VhTDMhVh-0}, \eqref{eq:VhTATDX}, and~\eqref{eq:DXTDX}, we have, together with~\eqref{eq:norm-AVi},
\begin{equation}  \label{eq:XTX}
\begin{split}
    \abs{X\trans X}_{ij}
    &\leq (\hat{S}_h)_{ij}
    + 2\biggl(\epseigh + \bigl(n^2\epsmh + \epseigmh\bigr) \kappa^2(B) \\
    &\quad+ 2n\sqrt{n}\bigl(\epsV + \epsU(1+\epsV)\bigr)\kappa(B)\biggr)(\tilde{\Sigma}_h)_{ii}(\tilde{\Sigma}_h)_{jj} \\
    &\leq (\hat{S}_h)_{ij} + \varepsilon_1(\tilde{\Sigma}_h)_{ii}(\tilde{\Sigma}_h)_{jj}.
\end{split}
\end{equation}
Using the definition of \(X\), i.e., \(X = \hat{U}\hat{\Sigma}\), and~\eqref{eq:sqrt}, then we have
\begin{equation*}
    \abs{\hat{U}\trans \hat{U}}_{ij}
    = \frac{\abs{X\trans X}_{ij}}{\hat{\sigma}_{i} \hat{\sigma}_{j}}
    \leq \frac{\abs{X\trans X}_{ij}}{(\tilde{\Sigma}_h)_{ii}(\tilde{\Sigma}_h)_{jj}(1-\epssqrt)^2}
    \leq \frac{1}{1-2\epssqrt}\biggl(\frac{(\hat{S}_h)_{ij}}{(\tilde{\Sigma}_h)_{ii}(\tilde{\Sigma}_h)_{jj}}
    + \varepsilon_1\biggr),
\end{equation*}
which draws the conclusion~\eqref{eq:thm:orthU1} by noticing that \(\hat{S}_h\) is a diagonal matrix satisfying~\eqref{eq:Sh=Sigma2}.
\end{proof}

\begin{remark}  \label{remark:eigen-jacobi}
    According to~\cite{H2002} and~\cite{DV1992}, \(\epsmh = \bigO(\machepsh)\), \(\epseigh = \bigO(\machepsh)\), \(\epssqrt = \bigO(\macheps)\), \(\epseigmh = \bigO(\machepsh)\), \(\epsV = \bigO(\macheps)\), and~\(\epsU = \bigO(\macheps)\).
    Based on Lemma~\ref{lem:algo2} and Theorem~\ref{thm:algo1}, we have the following conclusions for Algorithm~\ref{alg:mpthinsvd} using the two-sided Jacobi algorithm, i.e.,~\cite[Algorithm~3.1]{DV1992}, in~Line~\ref{line:eigen}:
    \begin{equation*}
        \hat{U}\hat{\Sigma}\tilde{V}_h\trans = A+\Delta A, \quad \norm{\Delta A(i, :)}\leq \bigO(\macheps)\norm{A(i, :)}.
    \end{equation*}
    If it further holds that
    \begin{equation*}
        \bigO(\machepsh)\kappa^2(B) \leq \frac{1}{2},
    \end{equation*}
    then
    \begin{equation*}
        \frac{\abs{\hat{\sigma}_i-\sigma_i(A)}}{\sigma_i(A)}
        \leq \bigO(\macheps) + \bigO(\machepsh)\kappa^2(B)
    \end{equation*}
    and
    \begin{equation*}
    \normF{\hat{U}\trans \hat{U}-I}
    \leq \bigO(\macheps) + \bigO(\machepsh)\kappa^2(B).
    \end{equation*}
\end{remark}

\subsection{Accuracy of the eigensolver in Algorithm~\ref{alg:eigen}}
\label{subsec:accuracy-eigen}
According to the standard rounding error analysis, each step of Algorithm~\ref{alg:eigen} satisfies that, respectively,
\begin{alignat}{2}
    \hat{M}_h + \Delta \hat{M}_h &= \hat{R}_h\trans \hat{R}_h, \quad &&\abs{\Delta \hat{M}_h}\leq \epscholh\abs{\hat{R}_h\trans}\abs{\hat{R}_h}, \label{eq:chol} \\
    \hat{R} &= \hat{R}_h + \Delta \hat{R}_h, \quad &&\abs{\Delta \hat{R}_h}\leq \macheps \abs{\hat{R}_h}, \label{eq:castR} \\
    \hat{R} + \Delta \hat{R} &= \tilde{U}_R \hat{\Sigma}\tilde{V}\trans, \quad &&\norm{\Delta \hat{R}(:, i)} \leq \epsSVD \norm{\hat{R}(:, i)}. \label{eq:svd}
\end{alignat}
To remain consistent with Section~\ref{subsec:accuracy-algo1}, we use \(\hat{M}_h\) as the input to Algorithm~\ref{alg:eigen}.
Based on these assumptions, we show that Algorithm~\ref{alg:eigen} satisfies~\eqref{eq:eig} and~\eqref{eq:eig-2} in Theorem~\ref{thm:algo2}.

\begin{theorem} \label{thm:algo2}
    Assume that Algorithm~\ref{alg:eigen} satisfies~\eqref{eq:chol}--\eqref{eq:svd} with given \(\epscholh\) and \(\epsSVD\).
    Let \(\hat{R}_h = B_RD_R\), where the columns of \(B_R\) have unit norms and \(D_R\) is a positive diagonal matrix whose entries are the column norms of \(\hat{R}_h\).
    If 
    \begin{equation} \label{eq:thm-algo2:assump}
        (\macheps + (1+\macheps)\epsSVD) \kappa(B_R) \leq \frac{1}{2},
    \end{equation}
    then the results of Algorithm~\ref{alg:eigen} satisfy \(\tilde{V}\trans(\hat{M}_h+\Delta\hat{M}_h)\tilde{V} = \hat{\Sigma}^2 + \Delta S\), where \(\Delta\hat{M}_h\) and \(\Delta S\) satisfy, respectively,
    \begin{align}
        \norm{D_R^{-1}\Delta \hat{M}_hD_R^{-1}} &\leq n\epscholh \norm{B_R}^2 \quad\text{and} \label{eq:thm-algo2:eig} \\
        \abs{\Delta S}_{ij} &\leq 6(\macheps + (1+\macheps)\epsSVD) \kappa(B_R) (\hat{\Sigma})_{ii} (\hat{\Sigma})_{jj},\quad \forall i, j\leq n.\label{eq:thm-algo2:eig-2}
    \end{align}
\end{theorem}

\begin{proof}
    Substituting \(\hat{R}\) involved in~\eqref{eq:svd} by~\eqref{eq:castR}, we have
    \begin{equation} \label{eq:proof-thm-algo2:hatRh}
        \hat{R}_h + \Delta \tilde{R}_h = \tilde{U}_R \hat{\Sigma}\tilde{V}\trans, \quad \norm{\Delta \tilde{R}_h(:, i)} \leq (\macheps + (1+\macheps)\epsSVD) \norm{\hat{R}_h(:, i)}
    \end{equation}
    with \(\Delta \tilde{R}_h := \Delta \hat{R}_h + \Delta \hat{R}\).
    Together with~\eqref{eq:chol}, we obtain
    \begin{equation} \label{eq:proof-thm-algo2:DS}
    \begin{split}
        \hat{M}_h + \Delta \hat{M}_h &= (\tilde{U}_R \hat{\Sigma}\tilde{V}\trans - \Delta \tilde{R}_h)\trans (\tilde{U}_R \hat{\Sigma}\tilde{V}\trans - \Delta \tilde{R}_h) \\
        &= \tilde{V} \hat{\Sigma} \tilde{V}\trans - \Delta \tilde{R}_h\trans \tilde{U}_R \hat{\Sigma}\tilde{V}\trans - \tilde{V}\hat{\Sigma}\tilde{U}_R\trans \Delta \tilde{R}_h
        + \Delta \tilde{R}_h\trans \Delta \tilde{R}_h \\
        &= \tilde{V} \hat{\Sigma} \tilde{V}\trans + \tilde{V}\bigl(\underbrace{-\tilde{V}\trans\Delta \tilde{R}_h\trans \tilde{U}_R \hat{\Sigma} - \hat{\Sigma}\tilde{U}_R\trans \Delta \tilde{R}_h\tilde{V}
        + \tilde{V}\trans\Delta \tilde{R}_h\trans \Delta \tilde{R}_h\tilde{V}}_{=: \Delta S}\bigr)\tilde{V}\trans.
    \end{split}
    \end{equation}
    
    We will first prove~\eqref{eq:thm-algo2:eig}.
    From~\eqref{eq:chol}, \(\norm{D_R^{-1} \Delta \hat{M}_h D_R^{-1}}\) can be bounded by
    \begin{equation*}
        \abs{D_R^{-1} \Delta \hat{M}_h D_R^{-1}}
        \leq \epscholh \abs{B_R\trans}\abs{B_R}.
    \end{equation*}
    Together with the inequality derived from~\cite[Lemma 6.6]{H2002}:
    \begin{equation*}
        \normbig{\abs{B_R\trans}\abs{B_R}} = \normbig{\abs{B_R}}^2 \leq n\norm{B_R}^2,
    \end{equation*}
    we have~\eqref{eq:thm-algo2:eig}.
    
    Then we will bound \((\Delta S)_{ij}\).
    From~\eqref{eq:proof-thm-algo2:hatRh}, \(-\tilde{V}\trans\Delta \tilde{R}_h\trans \tilde{U}_R \hat{\Sigma}\) can be bounded as
    \begin{equation} \label{eq:proof-thm-algo2:VhDRhURSR}
        \begin{split}
            \abs{-\tilde{V}\trans\Delta \tilde{R}_h\trans \tilde{U}_R \hat{\Sigma}}_{ij}
            &= \abs{\tilde{V}\trans \tilde{V} \hat{\Sigma}\tilde{U}_R\trans(\hat{R}_h + \Delta \tilde{R}_h)\itrans\Delta \tilde{R}_h\trans \tilde{U}_R \hat{\Sigma}}_{ij} \\
            &\leq \norm{(\hat{R}_h + \Delta \tilde{R}_h)\itrans\Delta \tilde{R}_h\trans} (\hat{\Sigma})_{ii} (\hat{\Sigma})_{jj}\\
            &\leq \norm{\Delta \tilde{R}_h\tilde{R}_h^{-1}} \cdot \norm{(I+\Delta \tilde{R}_h\tilde{R}_h^{-1})^{-1}}(\hat{\Sigma})_{ii} (\hat{\Sigma})_{jj} \\
            &\leq \frac{\norm{\Delta \tilde{R}_h\tilde{R}_h^{-1}}}{1-\norm{\Delta \tilde{R}_h\tilde{R}_h^{-1}}} (\hat{\Sigma})_{ii} (\hat{\Sigma})_{jj}.
        \end{split}
    \end{equation}
    Also, by~\eqref{eq:proof-thm-algo2:hatRh}, \(\norm{\Delta \tilde{R}_h\tilde{R}_h^{-1}}\) can be bounded by
    \begin{equation*}
        \norm{\Delta \tilde{R}_h\tilde{R}_h^{-1}}
        \leq (\macheps + (1+\macheps)\epsSVD) \kappa(B_R),
    \end{equation*}
    which gives the bounds of \(\abs{-\tilde{V}\trans\Delta \tilde{R}_h\trans \tilde{U}_R \hat{\Sigma}}_{ij}\) and \(\abs{- \hat{\Sigma}\tilde{U}_R\trans \Delta \tilde{R}_h\tilde{V}}_{ij}\), i.e.,
    \begin{align}
        \abs{-\tilde{V}\trans\Delta \tilde{R}_h\trans \tilde{U}_R \hat{\Sigma}}_{ij}
        &\leq \frac{(\macheps + (1+\macheps)\epsSVD) \kappa(B_R)}{1-(\macheps + (1+\macheps)\epsSVD) \kappa(B_R)} (\hat{\Sigma})_{ii} (\hat{\Sigma})_{jj}, \label{eq:proof-thm-algo2:DS-1} \\
        \abs{- \hat{\Sigma}\tilde{U}_R\trans \Delta \tilde{R}_h\tilde{V}}_{ij}
        &\leq \frac{(\macheps + (1+\macheps)\epsSVD) \kappa(B_R)}{1-(\macheps + (1+\macheps)\epsSVD) \kappa(B_R)} (\hat{\Sigma})_{ii} (\hat{\Sigma})_{jj}. \label{eq:proof-thm-algo2:DS-2}
    \end{align}
    Similarly, \(\tilde{V}\trans\Delta \tilde{R}_h\trans \Delta \tilde{R}_h\tilde{V}\) satisfies
    \begin{equation} \label{eq:proof-thm-algo2:DS-3}
        \abs{\tilde{V}\trans\Delta \tilde{R}_h\trans \Delta \tilde{R}_h\tilde{V}}
        \leq \biggl(\frac{(\macheps + (1+\macheps)\epsSVD) \kappa(B_R)}{1-(\macheps + (1+\macheps)\epsSVD) \kappa(B_R)}\biggr)^2 (\hat{\Sigma})_{ii} (\hat{\Sigma})_{jj}.
    \end{equation}
    Combining~\eqref{eq:proof-thm-algo2:DS} with the bounds~\eqref{eq:proof-thm-algo2:DS-1}, \eqref{eq:proof-thm-algo2:DS-2}, and~\eqref{eq:proof-thm-algo2:DS-3}, we prove~\eqref{eq:thm-algo2:eig-2} using the assumption~\eqref{eq:thm-algo2:assump}.
\end{proof}

In Theorem~\ref{thm:algo2}, we established the backward stability of Algorithm~\ref{alg:eigen}. Next, we will show that the analysis of Algorithm~\ref{alg:eigen} satisfies the assumptions~\eqref{eq:eig}--\eqref{eq:sqrt} stated in Lemma~\ref{lem:algo2}.

\begin{lemma} \label{lem:algo2}
    Assume that Algorithm~\ref{alg:eigen} satisfies~\eqref{eq:chol}--\eqref{eq:svd} with given \(\epscholh\) and \(\epsSVD\).
    If \(2(1+\epsmh)(\macheps + (1+\macheps)\epsSVD) \kappa(B) \leq \frac{1}{2}\) and
    \begin{equation} \label{eq:lem:assump}
        \biggl(\frac{(1+n\epsmh)2(n+2)\epscholh+4\epsmh}{1-2(n+3)\epscholh+4\epsmh} + n\epsmh\biggr) \kappa^2(B)\leq \frac{1}{2},
    \end{equation}
    then the results \(\hat{\Sigma}\) and \(\tilde{V}\) of Algorithm~\ref{alg:eigen} satisfy~\eqref{eq:eig}--\eqref{eq:sqrt} with \(\tilde{V}_h=\tilde{V}\) and \(\Sigma_h=\hat{\Sigma}\), where \(\epssqrt = 0\),
    \begin{equation*}
        \epseigmh = 2n(1+n\epsmh)\epscholh,\quad\text{and}\quad
        \epseigh = 12(1+n\epsmh)(\macheps + (1+\macheps)\epsSVD) \kappa(B).
    \end{equation*}
\end{lemma}

\begin{proof}
    Substituting \(\hat{M}_h\) in~\eqref{eq:chol} by~\eqref{eq:ATA}, we have
    \begin{equation} \label{eq:proof-lem-algo2:A-hatR}
        A\trans A + \Delta M + \Delta \hat{M}_h = \hat{R}_h\trans \hat{R}_h.
    \end{equation}
    Noticing that \(D_{ii} = (A\trans A)_{ii}\) and \((D_R)_{ii} = (\hat{R}_h\trans \hat{R}_h)_{ii}\), then \(D_{ii} + (\Delta M)_{ii} + (\Delta \hat{M}_h)_{ii} = (D_R)_{ii}\).
    Together with the bounds of \(\Delta M\) and  \(\Delta \hat{M}_h\) shown in~\eqref{eq:ATA} and~\eqref{eq:chol}, respectively, we have
    \begin{equation*}
        (1-\epsmh)D_{ii}
        \leq D_{ii} + (\Delta M)_{ii} = (D_R)_{ii} - (\Delta \hat{M}_h)_{ii}
        \leq (1+\epscholh)(D_R)_{ii},
    \end{equation*}
    which gives
    \begin{equation*}
        D_{ii}\leq \frac{1+\epscholh}{1-\epsmh}(D_R)_{ii}\quad\text{and similarly}
        \quad (D_R)_{ii}\leq \frac{1+\epsmh}{1-\epscholh} D_{ii}.
    \end{equation*}
    Noticing that \(D\) and \(D_R\) are diagonal matrix, then there exists a diagonal matrix \(\Delta D\) such that
    \begin{equation*}
        D_R=(I+\Delta D)D, \quad \abs{\Delta D}_{ii}\leq \frac{\epscholh+\epsmh}{1-\epscholh}.
    \end{equation*}
    Furthermore, by multiplicating \(D^{-1}\) from both right and left  sides of~\eqref{eq:proof-lem-algo2:A-hatR}, we can write~\eqref{eq:proof-lem-algo2:A-hatR} as
    \begin{equation} \label{eq:proof-lem-algo2:B-BR}
    \begin{split}
        &B\trans B + D^{-1}\Delta M D^{-1} \\
        &= (I+\Delta D)B_R\trans B_R(I+\Delta D) - (I+\Delta D)D_R^{-1}\Delta\hat{M}_hD_R^{-1}(I+\Delta D) \\
        &= \underbrace{\Delta DB_R\trans B_R + B_R\trans B_R\Delta D + \Delta DB_R\trans B_R\Delta D - (I+\Delta D)D_R^{-1}\Delta\hat{M}_hD_R^{-1}(I+\Delta D)}_{=: \Delta B_R} \\
        &\quad + B_R\trans B_R,
    \end{split}
    \end{equation}
    where, using~\eqref{eq:thm-algo2:eig}, \(\norm{\Delta B_R}\) can be bounded by
    \begin{equation*}
    \begin{split}
        \norm{\Delta B_R} &\leq 2\frac{\epscholh+\epsmh}{1-\epscholh}\norm{B_R}^2
        + \Bigl(\frac{\epscholh+\epsmh}{1-\epscholh}\Bigr)^2\norm{B_R}^2
        + \Bigl(\frac{1+\epsmh}{1-\epscholh}\Bigr)^2n\epscholh\norm{B_R}^2\\
        &\leq \biggl(\frac{2(n+2)\epscholh+4\epsmh}{1-2\epscholh}\biggr) \norm{B_R}^2.
    \end{split}
    \end{equation*}
    Using the perturbation theory on eigenvalues~\cite[Theorem~8.1.5]{GV2013} and~\eqref{eq:ATA}, we have
    \begin{equation*}
    \begin{split}
        \lambda_{\max}(B_R\trans B_R) - \biggl(\frac{2(n+2)\epscholh+4\epsmh}{1-2\epscholh}\biggr) \norm{B_R}^2
        &\leq \lambda_{\max}(B\trans B) + n\epsmh\norm{B}^2, \\
        \lambda_{\min}(B_R\trans B_R) + \biggl(\frac{2(n+2)\epscholh+4\epsmh}{1-2\epscholh}\biggr) \norm{B_R}^2
        &\geq \lambda_{\min}(B\trans B) - n\epsmh\norm{B}^2.
    \end{split}
    \end{equation*}
    Note that \(\lambda_{\max}(B_R\trans B_R) = \norm{B_R}^2\) and \(\lambda_{\max}(B\trans B) = \norm{B}^2\). 
    Together with~\eqref{eq:ATA} and the assumption~\eqref{eq:lem:assump}, \(\norm{B_R}\) can be bounded by
    \begin{equation} \label{eq:proof-lem-algo2:maxev}
        \norm{B_R}^2\leq \frac{1+n\epsmh}{1-\frac{2(n+2)\epscholh+4\epsmh}{1-2\epscholh}}\norm{B}^2
        \leq 2(1+n\epsmh)\norm{B}^2
    \end{equation}
    and further
    \begin{equation} \label{eq:proof-lem-algo2:minev}
        \lambda_n(B_R\trans B_R)\geq \lambda_n(B\trans B) - \biggl(\frac{(1+n\epsmh)2(n+2)\epscholh+4\epsmh}{1-2(n+3)\epscholh+4\epsmh} + n\epsmh\biggr) \norm{B}^2.
    \end{equation}
    Combining~\eqref{eq:proof-lem-algo2:maxev} with~\eqref{eq:proof-lem-algo2:minev} and the assumption~\eqref{eq:lem:assump}, we have
    \begin{equation} \label{eq:proof-lem-algo2:kappaBR}
    \begin{split}
        \kappa(B_R) &= \sqrt{\frac{\norm{B_R}^2}{\lambda_n(B_R\trans B_R)}}
        \leq \sqrt{4(1+n\epsmh)\kappa^2(B)} 
        \leq 2(1+n\epsmh)\kappa(B).
    \end{split}
    \end{equation}
    Substituting \(\norm{B_R}\) and \(\kappa(B_R)\) in Theorem~\ref{thm:algo2} by~\eqref{eq:proof-lem-algo2:maxev} and~\eqref{eq:proof-lem-algo2:kappaBR}, we conclude the proof.
\end{proof}

\begin{remark}
    According to~\cite{H2002} and~\cite{DV2008a}, \(\epsmh = \bigO(\machepsh)\), \(\epscholh = \bigO(\machepsh)\), \(\epsSVD = \bigO(\macheps)\), \(\epsV = \bigO(\macheps)\), and~\(\epsU = \bigO(\macheps)\).
    Based on Lemma~\ref{lem:algo2} and Theorem~\ref{thm:algo1}, we have the following conclusions for Algorithm~\ref{alg:mpthinsvd} using Algorithm~\ref{alg:eigen} in~Line~\ref{line:eigen}:
    \begin{equation*}
        \hat{U}\hat{\Sigma}\tilde{V}_h\trans = A+\Delta A, \quad \norm{\Delta A(i, :)}\leq \bigO(\macheps)\norm{A(i, :)}.
    \end{equation*}
    If it further holds that
    \begin{equation*}
        \bigO(\macheps)\kappa(B) \leq \frac{1}{2}\quad\text{and}\quad
        \bigO(\machepsh)\kappa^2(B) \leq \frac{1}{2},
    \end{equation*}
    then
    \begin{equation*}
        \frac{\abs{\hat{\sigma}_i-\sigma_i(A)}}{\sigma_i(A)}
        \leq \bigO(\macheps)\kappa(B) + \bigO(\machepsh)\kappa^2(B)
    \end{equation*}
    and
    \begin{equation*} 
    \normF{\hat{U}\trans \hat{U}-I}
    \leq \bigO(\macheps)\kappa(B) + \bigO(\machepsh)\kappa^2(B).
    \end{equation*}
\end{remark}

\subsection{Improved bound for the Cholesky QR algorithm}
As discussed in Section~\ref{sec:introduction}, Algorithm~\ref{alg:mpthinsvd} is motivated by the mixed precision Cholesky QR algorithm, which includes the following three steps:
\begin{enumerate}
    \item Compute the Gram matrix \(M = A\trans A\) at the precision \(\machepsh\).
    \item Compute the Cholesky factorization of \(M\) satisfying \(M = R\trans R\) at the precision \(\machepsh\).
    \item Compute \(Q\) by solving linear systems \(QR = A\) at the precision \(\macheps\).
\end{enumerate}
In~\cite[Theorem 3.2]{YTD2015}, for the QR factorization problem \(A = QR\), the authors analyzed the loss of orthogonality of the resulting orthonormal matrix \(\hat{Q}\) produced by the Cholesky QR algorithm, which is related to \(\kappa(A)\).
Inspired by the analysis of thin SVD algorithms, we now derive a sharper bound on this loss of orthogonality, as stated in the following theorem.

\begin{theorem}
    Let \(A = BD\), where the columns of \(B\) have unit norms and \(D\) is a diagonal matrix containing the column norms of \(A\).
    Assume that \(\hat{Q}\) is the \(Q\)-factor of \(A\) computed by the mixed precision Cholesky QR algorithm.
    
    If \(\machepsh=\bigO(\macheps^2)\) and \(\bigO(\macheps)\kappa(B)\leq 1\), then
    \begin{equation*}
        \norm{\hat{Q}\trans\hat{Q} - I}\leq \bigO(\macheps)\kappa(B).
    \end{equation*}
    
    If \(\machepsh=\macheps\) and \(\bigO(\macheps)\kappa^2(B)\leq 1\), then
    \begin{equation*}
        \norm{\hat{Q}\trans\hat{Q} - I}\leq \bigO(\macheps)\kappa^2(B).
    \end{equation*}
\end{theorem}

\begin{proof}
    According to the standard rounding error analysis, the Cholesky QR algorithm satisfies~\eqref{eq:ATA},
    \begin{equation} \label{eq:qr:chol}
         \hat{M}_h + \Delta \hat{M}_h = \hat{R}_h\trans \hat{R}_h, \quad \abs{\Delta \hat{M}_h}\leq \epscholh\abs{\hat{R}_h\trans}\abs{\hat{R}_h},
    \end{equation}
    and
    \begin{equation*}
        \hat{Q}(i, :)\bigl(\hat{R}_h + \Delta R_i\bigr) = A(i, :), \quad
        \abs{\Delta R_i}\leq \epstrsm\abs{\hat{R}_h},
    \end{equation*}
    which can be further written as, using \(A = BD\),
    \begin{alignat*}{2}
        &D\trans\hat{M}^{(B)}_h D = D\trans B\trans BD + D\trans \Delta M^{(B)}D, \quad &&\abs{\Delta M^{(B)}} \leq \epsmh\abs{B}\trans\abs{B}, \\
        &D\trans\hat{M}_h^{(B)}D + D\trans\Delta \hat{M}_h^{(B)}D = D\trans(\hat{R}_h^{(B)})\trans \hat{R}_h^{(B)}D, \quad &&\abs{\Delta \hat{M}_h^{(B)}}\leq \epscholh\abs{(\hat{R}_h^{(B)})\trans}\abs{\hat{R}_h^{(B)}}, \\
        &\hat{Q}(i, :)\bigl(\hat{R}_h^{(B)}D + \Delta R_i^{(B)}D\bigr) = B(i, :)D, \quad
        &&\abs{\Delta R_i^{(B)}}\leq \epstrsm\abs{\hat{R}_h^{(B)}}.
    \end{alignat*}
    By applying \(D\itrans\) and \(D^{-1}\) from the left and right sides, respectively, these equations can be written as
    \begin{alignat*}{2}
        &\hat{M}^{(B)}_h = B\trans B + \Delta M^{(B)}, \quad &&\abs{\Delta M^{(B)}} \leq \epsmh\abs{B}\trans\abs{B}, \\
        &\hat{M}_h^{(B)} + \Delta \hat{M}_h^{(B)} = (\hat{R}_h^{(B)})\trans \hat{R}_h^{(B)}, \quad &&\abs{\Delta \hat{M}_h^{(B)}}\leq \epscholh\abs{(\hat{R}_h^{(B)})\trans}\abs{\hat{R}_h^{(B)}}, \\
        &\hat{Q}(i, :)\bigl(\hat{R}_h^{(B)} + \Delta R_i^{(B)}\bigr) = B(i, :), \quad
        &&\abs{\Delta R_i^{(B)}}\leq \epstrsm\abs{\hat{R}_h^{(B)}}.
    \end{alignat*}
    This implies that the \(Q\)-factor of \(A\) produced by the Cholesky QR algorithm is identical to the \(Q\)-factor of \(B\) computed by the same procedure.
    Therefore, by applying the rounding error analysis for the Cholesky QR algorithm, as in~\cite{YNYF2015,YTD2015}, we draw the conclusion.
\end{proof}

\section{Numerical Experiments}
\label{sec:experiments}
In this section, we present numerical results for Algorithm~\ref{alg:mpthinsvd} to compute \(U\), \(\Sigma\), and \(V\) satisfying~\eqref{prob:svd}.
In our tests, the working precision is IEEE single precision and the higher precision is IEEE double precision.
Our MPI implementations use the \texttt{TSQR} routine provided by the Trilinos C++ library~\cite{trilinos2005}.
The experiments were carried out on the Karolina cluster,%
\footnote{https://docs.it4i.cz}
which consists of 720 nodes, each equipped with one 64-core CPU.
We use up to 256 nodes and employ one process
per node as the mode of parallelization.

\subsection{Tests on CPU}
On CPU, the following thin SVD variants are tested:
\begin{enumerate}
    \item {\bf mpthinSVD-Jacobi}: mixed precision thin SVD algorithm (Algorithm~\ref{alg:mpthinsvd}) with the one-sided Jacobi SVD algorithm (\texttt{DGEJSV}) as eigensolver in Line~\ref{line:eigen}.
    \item {\bf mpthinSVD-Algo 2}: mixed precision thin SVD algorithm (Algorithm~\ref{alg:mpthinsvd}) with Algorithm~\ref{alg:eigen} as eigensolver in Line~\ref{line:eigen}.
    \item {\bf Jacobi SVD (\texttt{sGEJSV})}: one-sided Jacobi SVD algorithm in LAPACK.
    \item {\bf QR SVD (\texttt{sGESVD})}: QR SVD algorithm in LAPACK.
    \item {\bf D\&C SVD (\texttt{sGESDD})}: D\&C SVD algorithm in LAPACK.
\end{enumerate}

Similar to~\cite{DV2008b,GMS2025}, the test matrices are randomly generated matrices of the form \(A = BD\), where \(B\) has columns of unit norm and \(D\) is diagonal.
We can specify the condition numbers of \(D\) and \(B\) in the following way.
First, we use \(\texttt{DLATM1}\) to generate two diagonal matrices, \(D\) and \(\Sigma\), with given condition numbers \(\kappa(D)\) and
\(\kappa(\Sigma)\), respectively.
Next, we form the matrix \(B\) as \(B = W_1 \Sigma W_2 W_3\), where
\(W_1\) and \(W_2\) are the unitary factors from the QR factorization of
random matrices, and \(W_3\) is selected appropriately so that
\(B\) has unit-norm columns~\cite{CL1983,DH2000}.
In our experiments, we test \(16\) combinations of \(\kappa(D)\) and
\(\kappa(B)\) as shown in Table~\ref{tab:testmat}.
The parameter \texttt{MODE} is used by \texttt{DLATM1} to specify different
types of test matrices; see Table~\ref{tab:XLATM1} for details.

\begin{table}[!tb]
\centering
\footnotesize
\caption{Description of different types of test matrices.}
\label{tab:testmat}
\begin{tabular}{|c|c|c|c|c|c|c|c|c|c|c|c|c|c|c|c|c|}
\hline
ID of matrices& 1& 2& 3& 4& 5& 6& 7& 8& 9& 10& 11&
12& 13& 14& 15& 16\\
\hline
\texttt{MODE} of \(D\)& 1& 1& 1& 1& 2& 2& 2& 3& 3& 3& 4&
4& 4& 5& 5& 5\\
\hline
\texttt{MODE} of \(\Sigma\)& 2& 3& 4& 5& 3& 4& 5& 2& 4& 5& 2&
3& 5& 2& 3& 4\\
\hline
\end{tabular}
\end{table}

\begin{table}[!tb]
\centering
\caption{Explanation of different modes of \texttt{XLATM1} for generating
\(D=\diag(d)\).}
\label{tab:XLATM1}
\begin{tabular}{cp{0.8\textwidth}}
\hline
\texttt{MODE} & \hfil Description\hfil \\
\hline
1 & Clustered at $1/\kappa$, where $d(1) = 1$, $d(2:n) = 1/\kappa$ \\
2 & Clustered at $1$, where $d(1:n-1) = 1$, $d(n) = 1/\kappa$ \\
3 & Geometric distribution, where $d(i) = 1/\kappa^{(i-1)/(n-1)}$ \\
4 & Arithmetic distribution, where $d(i) = 1 - (i-1)/(n-1)(1 - 1/\kappa)$ \\
5 & Log-random distribution, where $d(i)\in(1/\kappa,1)$ so that their logarithms are random and uniformly
distributed \\
\hline
\end{tabular}
\end{table}

\subsubsection{Tests for accuracy}
In Figure~\ref{fig:accuracy}, we consider five different scenarios for \(\kappa(B)\): \(10\), \(10^2\), \(10^3\), \(10^4\), and \(10^5\).
For each value of \(\kappa(B)\), we fix \(n = 64\) and \(m = 16\cdot n\), and choose \(\kappa(D)\) from \(\{1, 10^2, 10^4, 10^6, 10^8\}\).
For every pair \((\kappa(D), \kappa(B))\), we evaluate the \(16\) configurations listed in Table~\ref{tab:testmat}.

Figure~\ref{fig:accuracy} presents the maximum relative error of the singular values,
\[
\max_{i\leq n}\frac{\abs{\hat{\sigma}_i-\sigma_i}}{\sigma_i},
\]
where \(\hat{\sigma}_i\) denotes the singular values computed by the different algorithms, and \(\sigma_i\) denotes the reference singular values obtained using the one-sided Jacobi SVD algorithm in double precision, i.e., \texttt{DGEJSV}. 
The figure clearly demonstrates that our mixed precision thin SVD algorithm achieves higher accuracy than both the QR SVD and the D\&C SVD algorithms, and achieves accuracy comparable to the Jacobi SVD algorithm, consistent with our theoretical prediction indicated by the blue dotted line.
As indicated in Remark~\ref{remark:eigen-jacobi}, when \(\kappa(B)\) is relatively small, the dominant term in the bound for the singular values is \(\bigO(\macheps)\), since \(\macheps < \machepsh\kappa^2(B)\).
Consequently, in the first two subfigures, the blue curve representing the relative error of mpthinSVD-Jacobi stays at about \(\bigO(\macheps)\) instead of \(\bigO(\macheps)\kappa(B)\), and it also outperforms the red curve for mpthinSVD-Algo 2, whose behavior is dominated by \(\bigO(\macheps)\kappa(B)\).
As \(\kappa(B)\) increases, the accuracy of the two algorithms becomes comparable.

\begin{figure}[!tb]
    \centering
    \includegraphics[width= \textwidth]{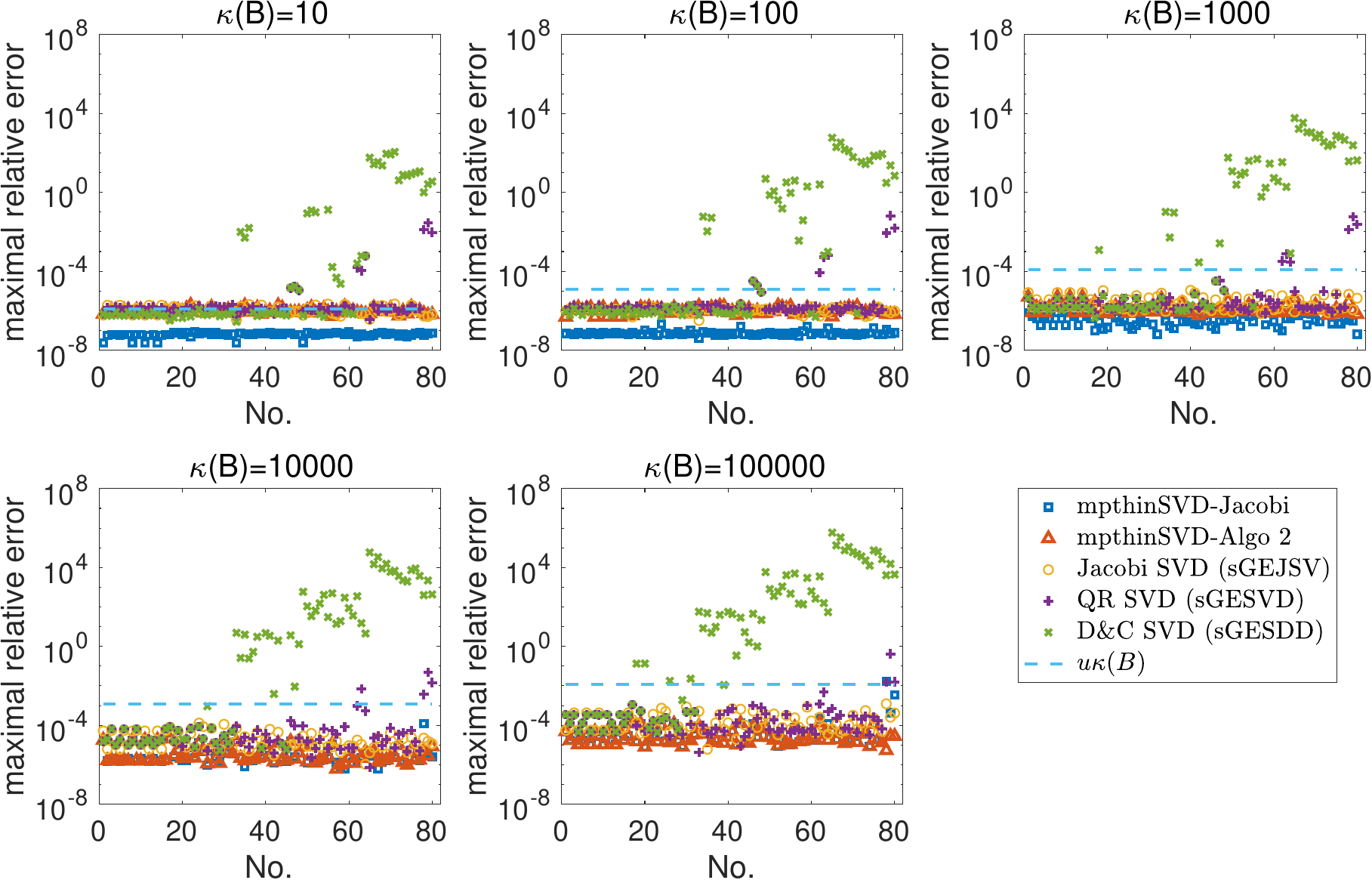}
    \caption{Accuracy comparison among different algorithms for computing SVD of tall-and-skinny matrices. In this figure, the results for \(80\) test matrices are presented, labeled as No.1 through No.80 along the \(x\)-axis.}
    \label{fig:accuracy}
\end{figure}

\subsubsection{Tests for performance on one CPU}
For performance tests, we consider random matrices of various dimensions.
In Figure~\ref{fig:performance}, the four subfigures correspond to fixed values \(n=16\), \(32\), \(64\), and \(128\), respectively.
For each choice of \(n\), we test four settings of \(m\), namely \(m = 32\cdot n\), \(256\cdot n\), \(2048\cdot n\), and \(16{,}384\cdot n\).
Each subfigure reports the relative run time, defined as the ratio of the wall clock time of a given algorithm over that of the Jacobi SVD algorithm.
We use the labels `mpthinSVD-Jacobi:Gram matrix'\slash`mpthinSVD-Algo 2:Gram matrix', `mpthinSVD-Jacobi:eigen'\slash`mpthinSVD-Algo 2:eigen', `mpthinSVD-Jacobi:compute U'\slash`mpthinSVD-Algo 2:compute U', and `mpthinSVD-Jacobi:overlap'\slash`mpthinSVD-Algo 2:overlap' to denote the different components of Algorithm~\ref{alg:mpthinsvd} when combined with the two eigensolvers (the one-sided Jacobi SVD algorithm or Algorithm~\ref{alg:eigen}), respectively: 
forming the Gram matrix,
computing the spectral decomposition of the Gram matrix,
computing \(U\),
and the remaining parts of the algorithm.

Figure~\ref{fig:performance} illustrates that mpthinSVD-Algo 2 is slightly faster than mpthinSVD-Jacobi, because Algorithm~\ref{alg:eigen} outperforms the one-sided Jacobi algorithm.
This performance gap widens as \(m\) increases, and when \(m/n=16{,}384\), Algorithm~\ref{alg:mpthinsvd} can be more than \(8\times\) faster than the QR-based algorithms.

\begin{figure}[!tb]
    \centering
    \includegraphics[width= \textwidth]{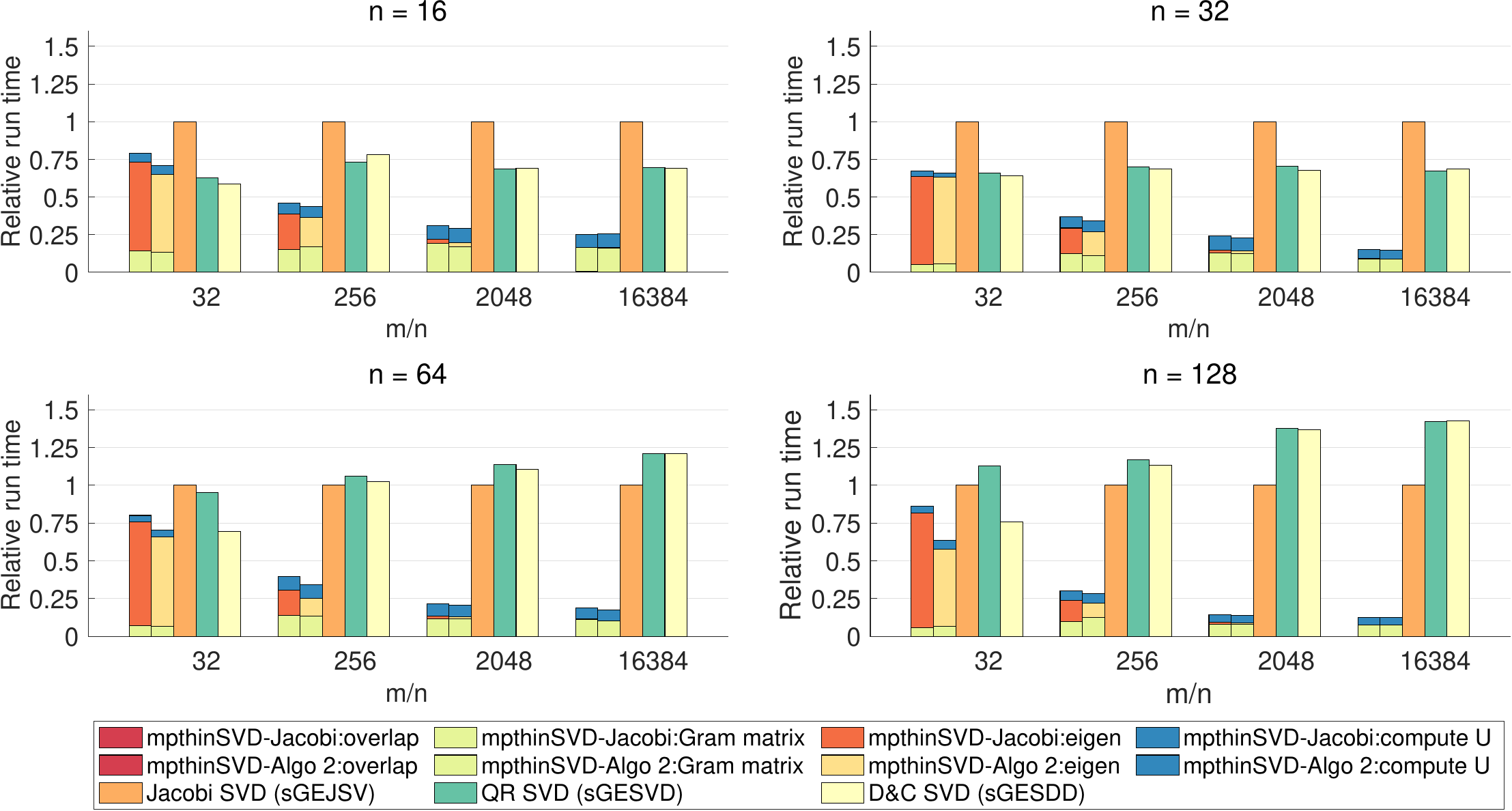}
    \caption{Performance comparison on CPU among different algorithms for computing SVD of tall-and-skinny random matrices. For each ratio \(m/n\), the five bars from
    left to right represent the result of mpthinSVD-Jacobi, mpthinSVD-Algo 2, Jacobi SVD (\texttt{sGEJSV}), QR SVD (\texttt{sGESVD}), and D\&C SVD (\texttt{sGESDD}), respectively.}
    \label{fig:performance}
\end{figure}

\subsection{Tests on the distributed memory system}
\subsubsection{MPI implementation}
\label{subsec:mpi}
Before presenting the experiments, we outline the MPI parallelization strategy for Algorithm~\ref{alg:mpthinsvd}.
We consider a setup with \(p\) nodes, each running a single MPI process, and adopt a 1D rowwise data distribution for both the input matrix \(A\in\mathbb{R}^{m\times n}\) and the left singular vectors \(U\in\mathbb{R}^{m\times n}\).
Furthermore, we assume that each process has sufficient memory to store matrices of size \(n\times n\) and \(\lceil m/p\rceil\times n\).
All processes hold local copies of the small matrices \(M_h\), \(\Sigma_h\), \(V_h\), \(\Sigma\), and \(V\).
Within Algorithm~\ref{alg:mpthinsvd}, the only operation that requires global communication is the matrix product of two tall-and-thin matrices, namely \(A_h\trans A_h\).
Each process first computes the product of its local matrix blocks, and the final result is then assembled by means of a single global synchronization step.
\subsubsection{Tests for performance}
Similarly to the performance tests on CPU, we also consider random matrices of various dimensions.
From the tests on CPU, when the matrix size becomes larger, the Gram-based and QR-based algorithms are dominated by the computational time to form the Gram matrix and compute the QR factorization, respectively.
Thus, we only compare the following variants:
\begin{enumerate}
    \item {\bf{mpthinSVD-Algo 2}}: mixed precision thin SVD algorithm (Algorithm~\ref{alg:mpthinsvd}) with Algorithm~\ref{alg:eigen} as eigensolver in Line~\ref{line:eigen}.
    The MPI implementation is shown in Section~\ref{subsec:mpi}.
    \item {\bf Jacobi SVD}: using \texttt{TSQR} to reduce the tall-and-skinny matrix to a small square matrix and then employing one-sided Jacobi SVD algorithm to solve the SVD of the small matrix, as discussed in Section~\ref{sec:introduction}.
\end{enumerate}

Figures~\ref{fig:performance-fixed-m} and~\ref{fig:performance-fixed-msub} report the speedups for the strong- and weak-scaling experiments, respectively, where the speedup is defined as the ratio between the wall clock time of the Jacobi SVD algorithm and that of the mpthinSVD-Algo 2 algorithm.
For a relatively small number of nodes, the runtime is mainly determined by the computational cost, data movement, and cache misses, leading to larger speedups.
As the node count increases, communication overhead, i.e., global synchronization, becomes dominant, and the speedup plateaus at roughly \(2\times\).
This is because our mixed-precision thin SVD algorithm requires only a single synchronization point, whereas \texttt{TSQR} may require additional synchronization overhead to form the \(Q\)-factor explicitly.

\begin{figure}[!tb]
    \centering
    \includegraphics[width= 0.6\textwidth]{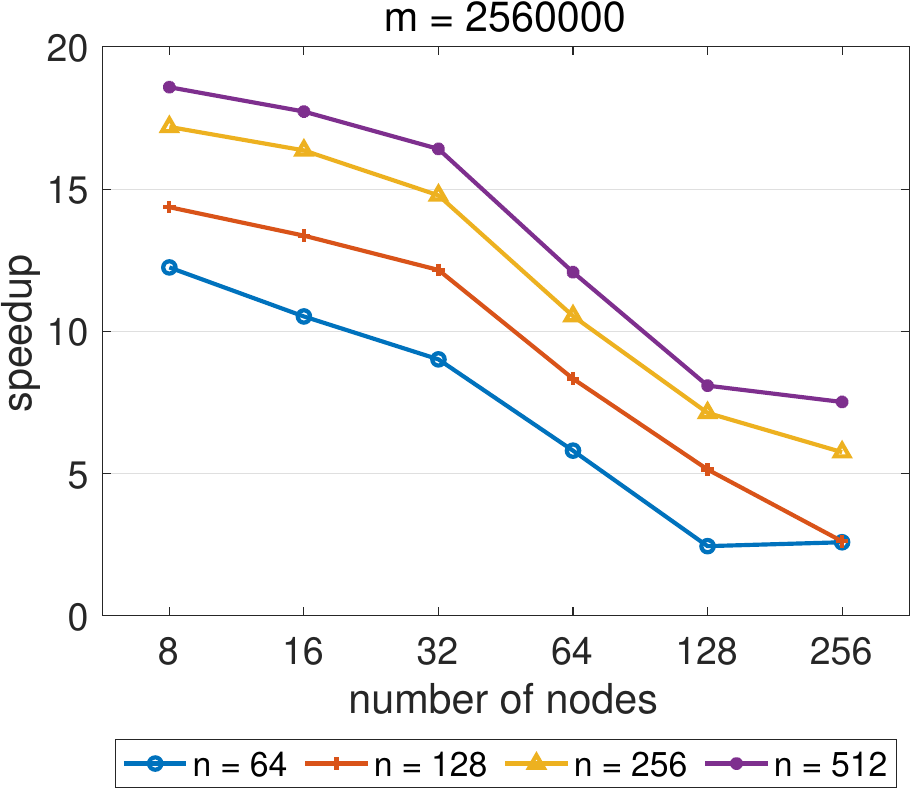}
    \caption{Performance comparison on distributed memory system between different algorithms for computing SVD of tall-and-skinny random matrices with fixed number of rows, i.e., \(m\).}
    \label{fig:performance-fixed-m}
\end{figure}

\begin{figure}[!tb]
    \centering
    \includegraphics[width= \textwidth]{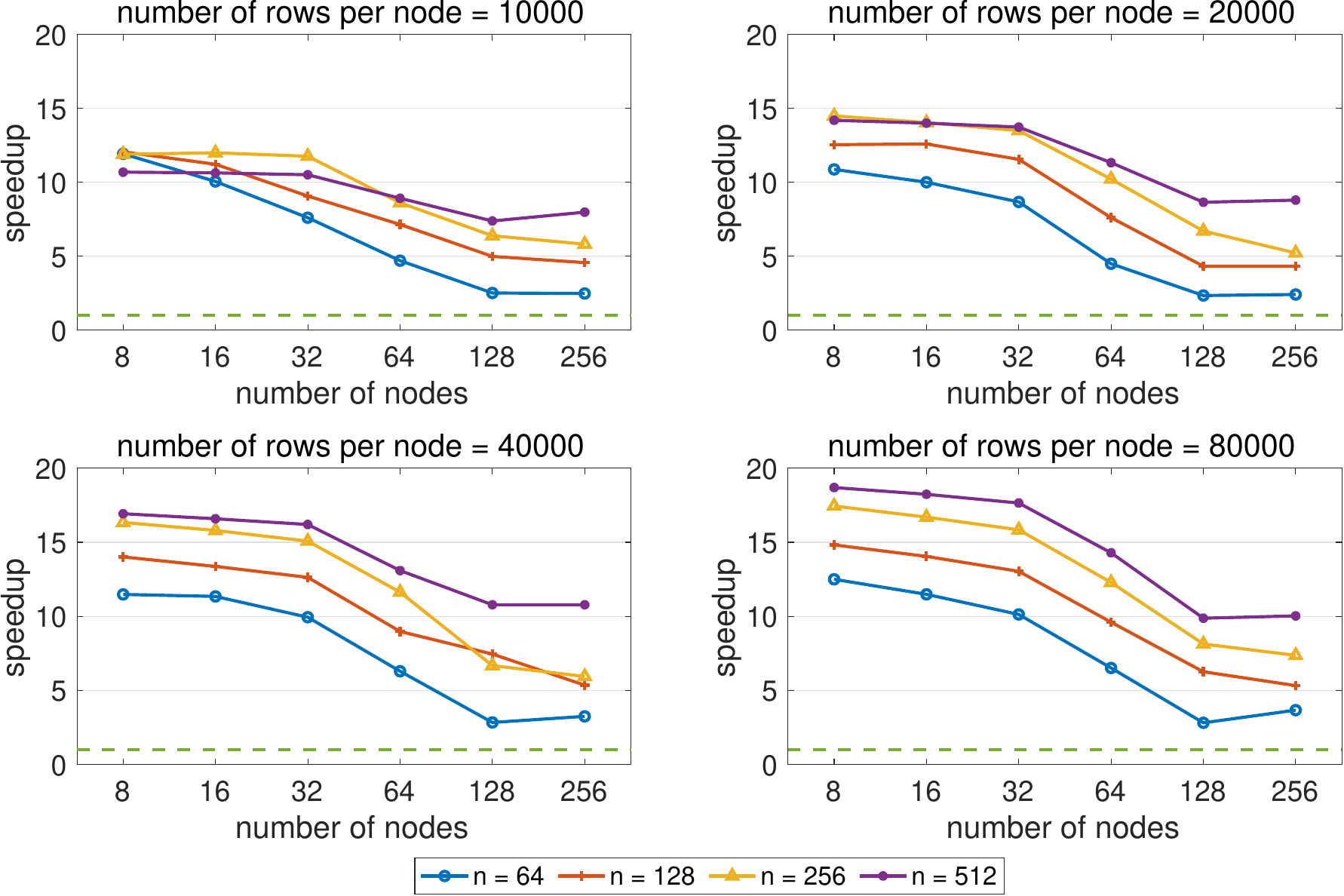}
    \caption{Performance comparison on distributed memory system between different algorithms for computing SVD of tall-and-skinny random matrices with fixed number of rows per node, i.e., \(m/p\), where \(p\) denotes the number of nodes.}
    \label{fig:performance-fixed-msub}
\end{figure}

\section{Conclusions}
\label{sec:conclusions}
In this paper, we introduce a mixed precision thin SVD algorithm based on the higher precision Gram matrix and the Jacobi algorithm, and we prove that the singular values produced by this algorithm achieve high relative accuracy.
Numerical experiments demonstrate that, on both shared memory and distributed memory systems, the mixed precision algorithm significantly outperforms the traditional QR-based thin SVD approaches.

We note that randomized QR-based thin SVD algorithms also exist, but in this work we restrict our attention to deterministic methods that achieve high relative accuracy.
In future work, it would be interesting to incorporate mixed precision techniques into these randomized thin SVD algorithms and compare their performance with our mixed precision algorithm.
\section*{Acknowledgments}
E.~Carson and Y.~Ma are supported by the European Union (ERC, inEXASCALE, 101075632).
Views and opinions expressed are those of the authors only and do not necessarily reflect those of the European Union or the European Research Council.
Neither the European Union nor the granting authority can be held responsible for them.
E.~Carson and Y.~Ma additionally acknowledge support from the Charles University Research Centre program No. UNCE/24/SCI/005.
M.~Shao is supported by the National Key R\&D Program of China under
Grant Number 2023YFB3001603.
We also acknowledge IT4Innovations National Supercomputing Center, Czech Republic, for awarding this project (Project FTA-25-81) access to the Karolina supercomputer, supported by the Ministry of Education, Youth and Sports of the Czech Republic through the e-INFRA CZ (grant ID: 90254).

\bibliographystyle{abbrvurl}
\bibliography{mybib}


\end{document}